\documentclass[11pt, reqno]{amsart}
\usepackage{amssymb,rotating,hyperref}

\newtheorem{theorem}{Theorem}[section]
\newtheorem{lemma}[theorem]{Lemma}

\newtheorem{proposition}[theorem]{Proposition}

\newtheorem*{lconjecture}{Lang's Conjecture}

\theoremstyle{definition}

\newtheorem{definition}[theorem]{Definition}

\newtheorem{conjecture}[theorem]{Conjecture}
\newtheorem{example}[theorem]{Example}

\newtheorem{remark}[theorem]{Remark}

\def\lcm{\operatorname{lcm}}

\def\Z{\mathbb{Z}}
\def\Q{\mathbb{Q}}
\def\N{\mathbb{N}}

\def\beq {\begin{equation}}
\def\endq {\end{equation}}

\newcommand{\QQ}{\mathbb{Q}}

\newcommand{\ZZ}{\mathbb{Z}}
\newcommand{\h}{\hat{h}}
\newcommand{\rank}{\operatorname{rank}}

\newcommand{\Ocal}{\mathcal{O}}

\newcount\mazur
\def\mazurbound{\number\mazur}
\newcount\qt
\def\qtbound{\number\qt}

\renewcommand{\epsilon}{\varepsilon}
\def\ge{\geqslant}
\def\geq{\geqslant}
\def\le{\leqslant}
\def\leq{\leqslant}

\begin{document}
\title{The uniform primality conjecture for elliptic curves}
\subjclass{11G05, 11A41} \keywords{Canonical height, divisibility
sequence, division polynomial, elliptic curve, isogeny, prime,
rational function field, reduction, Siegel's Theorem}
\thanks{The research of the second author was supported by grant
from NSERC of Canada and one from the LMS. The research of the third author was
supported by a grant from EPSRC}
\author{Graham Everest, Patrick Ingram, Val\' ery Mah\'e, Shaun Stevens.}
\address{(GE,VM+SS) School of Mathematics, University of East Anglia,
Norwich NR4 7TJ, UK}
\address{(PI) Department of Mathematics,
University of Toronto, Canada M5S 2E4}

\email{g.everest@uea.ac.uk}
\email{pingram@math.utoronto.ca}
\email{v.mahe@uea.ac.uk}
\email{shaun.stevens@uea.ac.uk}

\begin{abstract}
An elliptic divisibility sequence, generated by a point in the
image of a rational isogeny, is shown to possess a uniformly
bounded number of prime terms. This result applies over the
rational numbers, assuming Lang's conjecture, and over the
rational function field, unconditionally. In the latter case, a
uniform bound is obtained on the index of a prime term. Sharpened
versions of these techniques are shown to lead to explicit results
where all the irreducible terms can be computed.
\end{abstract}

\maketitle


\section{Introduction}\label{intro}

The Mersenne problem asks if infinitely many integers of the form $2^n-1$ are prime.
More generally, if $a>b$ are positive coprime integers, one can ask if the
sequence
\begin{equation}\label{eq:mersenne}
V_n=\left(\frac{a^n-b^n}{a-b}\right)_{n\geq 1}
\end{equation}
has
infinitely many prime terms.  The answer is negative if $a/b$ is a
perfect power. For example, if $a=A^2$ and $b=B^2$, with $A$ and
$B\in\ZZ$, then the terms factorize as
$$V_n=\left(\frac{A^n-B^n}{A-B}\right)\left(\frac{A^n+B^n}{A+B}\right)$$
for $n$ odd and
$$V_n=\left(\frac{A^n-B^n}{A^2-B^2}\right)\left(A^n+B^n\right)$$
for $n$ even. It is easy to see that neither factor may be a unit
when $n$ is large enough.  Indeed, since the greatest common
divisor of the two terms on the right (in either expression) is
easily controlled, there are few possibilities for $V_n$ to be a
prime power, a remark germane to this paper.

Let $E/\QQ$ denote an elliptic curve, given by a Weierstrass equation with
integer coefficients, and let $P\in E(\QQ)$ denote a non-torsion point.  One can always write
$$nP=\left(\frac{A_n}{B_n^2}, \frac{C_n}{B_n^3}\right),$$
with $A_n, B_n$, and $C_n\in\ZZ$, and $\gcd(A_n, B_n)=1$.  The
sequence $(B_n)_{n\geq 1}$ is an \emph{elliptic divisibility
sequence}. These sequences have been the focus of attention of a
number of authors: consult
\cite{cz,eds,pe,primeds,emw,pi_eds,bjorn,shipsey-thesis,
sileds1,sileds2,sileds3,streng,mow}. The role of an elliptic
divisibility sequence relative to the addition law of $E(\QQ)$ is
analogous to the role of the sequence \eqref{eq:mersenne} relative
to the multiplicative group of $\QQ$. It is unlikely that $B_n$
will ever be prime, because of the relation $B_1\mid B_n$. It is
natural, as in the multiplicative case, to ask when $B_n/B_1$
might be prime, or a prime power.

\subsection{Background}

Chudnovsky and Chudnovsky \cite{chuds} considered the question of
prime occurrence in elliptic divisibility sequences
computationally, finding some examples of prime values which are
large (several hundred decimal digits) in comparison with the
starting data. These computations were not easy on several counts.
For one thing, the terms grow very rapidly - the logarithm is
quadratic in the index, and proportional to the global canonical
height. Also, there was a shortage of known small height points.
Finally, the computing power at the time was much less than what
is available today. With the examples to hand, searching was
feasible only out to index around $n=100$. Much larger prime terms
in elliptic divisibility sequences have been found recently.

\begin{example}\label{brid}
Consider $$E : y^2+y=x^3-x \mbox{ with } P=[0,0].$$
The following table shows indices $n$ for which $B_n$ is prime (actually, a Miller-Rabin
pseudo-prime to
10 randomly chosen bases), together with the number $\# B_n$ of decimal digits of $B_n$:

\medskip
\begin{tabular}{|c||c|c|c|c|c|c|c|c|c|c|c|c|c|c|}
\hline $n$ &5 &  7 & 8 & 9 & 11 & 12 & 13 &  19 & 23 &  29 &  83 &
101 & 409 & 1291 \\ \hline $\# B_n$ &1
&1&1&1&2&2&2&4&6&10&77&114&1857&18498\\\hline
\end{tabular}
\medskip

\noindent The final two primes were found by Br\'id N\'i Fhlathu\'in (1999) and the
first author (2006) respectively using
MAGMA \cite{magma} and PARI-GP \cite{parigp}.
The largest takes a couple of hours to test for (pseudo-) primality.
\end{example}

Bearing in mind that all of these examples have appeared at the
outer limits of what is computationally feasible, at the time,
this might suggest that some elliptic divisibility sequences
contain infinitely many prime terms. On the other hand, in
\cite{eds}, all of the examples from \cite{chuds} were tested out
to $n=500$ but no more prime terms appeared. Furthermore, while it
is true that the primes in Example~\ref{brid} are large, this
sequence has, up to the present, produced only 14 prime terms. It
also manifests a pattern witnessed in several examples (see the
table in section~\ref{comps}), namely, that a \emph{gap principle}
seems to be at work. What this means is that the gaps between
prime terms grow quickly in proportion to the index. Taking this
together with the quadratic exponential growth rate of the
sequence forces any outlying prime term to be inordinately large.
In section~\ref{comps}, the results of extensive computations are
exhibited, showing the appearance of prime terms in elliptic
divisibility sequences generated by the 18 rational points with
smallest known global canonical height in Elkies'
table~\cite{nde}. These, together with all the computations
performed for this paper, used both \cite{magma} and
\cite{parigp}. Together with a heuristic argument from~\cite{eds}
(revisited in section~\ref{HA}), this indicates that the case of
elliptic divisibility sequences could be quite different from that
of the multiplicative sequences \eqref{eq:mersenne}, and lends
support to what we believe could be true, namely:

\begin{conjecture}\label{Qupc}If $B=(B_n)$ denotes an
elliptic divisibility sequence generated by a rational point on an
elliptic curve in minimal form, then the number of prime terms
$B_n/B_1$ is uniformly bounded, independent of curve and point.
\end{conjecture}

In \cite{primeds} it was
demonstrated that if $P\in E(\QQ)$ lies in the image of a non-trivial
isogeny 
then $B_n$ can be prime (indeed a prime power) for only finitely many~$n$.
When the isogeny is the endomorphism~$[k]$ (multiplication by the integer~$k$), this
is in direct analogy with the trivial
case for sequences of the form \eqref{eq:mersenne}, discussed earlier.
The proof follows by giving a preliminary
argument showing the existence of a canonical factorization, then
estimates arising from a strong form of Siegel's Theorem
show that each factor is
too large to be a unit.

In the current article, a proof of Conjecture~\ref{Qupc} will be provided,
assuming Lang's Conjecture, in the case where
$P$ lies in the image of a non-trivial isogeny. The methods include the use of a
quantitative version of Siegel's Theorem proved by Silverman.
As the question entails the application of results from diophantine approximation, it is also
interesting to
consider the problem over the function field $\QQ(T)$, where stronger diophantine
approximation results are known (for example, Lang's Conjecture). Stronger conclusions do indeed become possible.

\subsection{Main Results}\label{results}
Let $K$ be $\QQ$ or $\QQ(t)$, and let $\mathcal O_K$ be $\ZZ$ or $\QQ[t]$ respectively.
Let $E$ be an elliptic curve over $K$ and $P\in E(K)$; then we can write
$$
x(P)=\frac{A_P}{B_P^2},
$$
with $A_P,B_P\in\mathcal O_K$ coprime and $B_P>0$ in the case $K=\QQ$, or $B_P$ monic in the case $K=\QQ(t)$.

Much of the discourse assumes a conjecture of Lang, which arises
naturally, and appears to be a necessary assumption in any attempt
to solve Conjecture~\ref{Qupc}. This conjecture will now be stated
and relies upon definitions provided in full in
section~\ref{notation}. It is known that the canonical height
$\h(P)$ is zero if and only if $P$ is a point of finite order, so
it is natural to ask how small $\h(P)$ might be for non-torsion
points $P$.  This question turns out to be quite important:  in
general, quantitative estimates for diophantine approximation on
elliptic curves all rely on some sort of lower bound on
$\h(P)/h(E)$, where $h(E)$ is the height of the curve. Considering
elliptic curves in general, it is not hard to show that
$\h(P)/h(E)$ can be made arbitrarily small (without resorting to
choosing torsion points $P$). For minimal curves $E/K$, however,
this seems not to be the case.

\begin{lconjecture}
There exists $\delta>0$, which depends only on $K$, such that $\h(P)\geq \delta\max\{1,h(E)\}$,
for all minimal curves $E/K$ and all non-torsion points $P\in E(K)$.
\end{lconjecture}

Our main results use a definition which appeared first in \cite{primeds} and in
a more general form in \cite{pe}. In \cite{primeds}, some crude counting
amongst small conductor curves suggested that when $K=\mathbb Q$, the definition
applies in a very
large number of non-trivial cases.

\begin{definition} If $P\in E(K)$ is non-torsion and is the image of a $K$-rational point
under a non-trivial $K$-rational isogeny, then say that $P$ is
\emph{magnified}.
\end{definition}

By composing the pre-image of the isogeny with an
isomorphism if necessary, it
may always be arranged that the pre-image is in minimal form. This will be a standing
assumption
throughout the paper.

In the rational case, the main theorem of the paper follows (see
Theorem~\ref{th:main} for a more precise version):

\begin{theorem}\label{mainQ}Assume $K=\QQ$ and $B$ is an elliptic
divisibility sequence generated by a magnified $K$-rational point
on an elliptic curve in minimal form. If Lang's conjecture holds
then the number of prime power terms $B_{n}/B_1$ is bounded
independently of the curve, the point and the degree of the
isogeny.
\end{theorem}

When $K=\QQ(t)$, as anticipated, an unconditional version of
Theorem~\ref{mainQ} holds. A bigger surprise is the different
character of the main conclusion.

\begin{theorem}\label{mainQt}Assume $K=\QQ(t)$ and $B$ is an
elliptic divisibility sequence arising from
a magnified $K$-rational point on an elliptic curve in minimal form.
Then $B_{n}/B_1$ fails to be a prime power for all indices $n\ge\qtbound$.
\end{theorem}

The different nature of the two results is worthy of comment and
sections \ref{FFcase} and \ref{Ratcase} are given over to a full
discussion of the different cases. At this point, the following
question seems worth asking, it is one that we have not been able
to resolve.

\noindent\textbf{Question} In the case when $K=\QQ$, assume that
$B$ is an elliptic divisibility sequence generated by a magnified
$K$-rational point on an elliptic curve in minimal form. If Lang's
conjecture is true, does it follow that $B_n/B_1$ fails to be a
prime power for all~$n$ beyond some uniform bound?

A positive answer to this question is certainly desirable, because
it means that a large class of explicit examples can, at least
hypothetically, be tested to find all the prime terms. With
current techniques, explicit examples do exist~\cite{emr} (see
also Example~\ref{cncs}), but for a limited class of curves.
Theoretically too, a positive answer is satisfying because it
means that the two cases when $K=\QQ$ and $K=\QQ(t)$ are exactly
parallel, assuming the generating point $P$ is magnified. Note
however that, over $K=\QQ(t)$, the primality conjecture itself is
not likely to be true - a discussion, which includes some striking
examples, follows in section~\ref{FFcase}.

\medskip

In the next section, a proof will be given that, under Lang's conjecture, a uniform bound exists
for the number of prime terms~$B_n/B_1$ which applies in both the rational
and the function field cases, assuming the elliptic divisibility sequence
is generated by a magnified point on a minimal curve. This section highlights
the convergence of the two theories. In sections~\ref{FFcase} and~\ref{Ratcase},
which follow on, various
points of divergence will be discussed, often with reference to explicit examples.

\section{The Uniform Primality Conjecture in Both Cases}\label{both}
\subsection{Notation}\label{notation}
The following notation will be standard throughout:
\medskip

$\begin{array}{cl}
K & \textrm{ either $\Q$ or $\Q (t)$;}\\[3pt]
\mathcal{O}_{K} & \textrm{ either $\Z$ if $K=\Q$ or $\Q[t]$ if
$K=\Q(t)$;}\\[3pt]
v &\textrm{ a prime of $K$ with completion $K_v$, ring of integers $\Ocal_v$;}\\[3pt]
\log |.|_{\infty}& \textrm{ either $\log (|.|)$ if $K=\Q$ or $\deg$
if $K=\Q (t)$;}\\[3pt]
h(\frac{p}{q}) & \textrm{ the height of $\frac{p}{q}$
defined, for $p,q\in\mathcal{O}_{K}$ coprime,
by }\\[3pt]
& \textrm{ }\qquad\quad h(\frac{p}{q}) :=\max (\log |p|_{\infty}, \log
|q|_{\infty});\\[3pt]
E & \textrm{ an elliptic curve over $K$}\\[3pt]
\Delta\textrm{ or }\Delta_{E} & \textrm{ the discriminant of
$E$;}\\[3pt]
j\textrm{ or }j_{E} & \textrm{ the $j$-invariant of $E$;}\\[3pt]
h(E) & \textrm{ the height $h(E):=\frac{1}{12}\max (h(j),h(\Delta ))$
of $E$;}\\[3pt]
P& \textrm{ a point of $E(K)$;}\\[3pt]
h(P) & \textrm{ the height $h(P):=\frac{1}{2}h(x(P))$ of $P$}\\[3pt]
\widehat{h}(P) & \textrm{ the canonical height $\widehat{h}(P)
:=\displaystyle\lim_{n\longrightarrow\infty}\tfrac{h(nP)}{n^{2}}$ of $P$;}
\end{array}$

\medskip

Generally speaking, when the existence of a uniform constant is
postulated, what is meant is a constant independent of the choice
of the elliptic curve (and the point studied if there is one). Of
course, such a constant may depend on the choice of the base
field~$K$.

\subsection{Reduction to Siegel's Theorem}

For ease of exposition, define the \emph{Lang ratio} of $P\in
E(K)$ to be $\rho(P, E)=\h(P)/\max\{1, h(E)\}$. Then Lang's
conjecture says that there exists $\delta>0$ such that
$$
\rho(P, E)\geq \delta,
$$
for all minimal curves $E/K$ and all non-torsion points $P\in E(K)$.
Silverman~\cite{silvermansiegel} has shown,
for number fields $K$, that $\rho(P, E)$ may be bounded
below in terms of the number of primes at which $E/K$ has split
multiplicative reduction.  Expanding
on these ideas, Hindry and Silverman \cite{hindrysilverman} showed
that $\rho(P, E)$ may be bounded
below in terms of an upper bound on the Szpiro ratio of $E/K$, that is, the ratio of the log-discriminant
of $E$ to the log-conductor.  Hindry and Silverman also showed that if
$K$ is a one-dimensional function
field, then the Szpiro ratio of an $S$-minimal elliptic curve (where $S$ is
a set of primes of $K$) is bounded
above absolutely in terms of $S$;
in particular, with $S=\{\infty\}$, Lang's Conjecture holds. In this
paper, it will
often be convenient to work in terms of $\rho(P, E)$ in order to
obtain unconditional results.

The following is a more precise version of Theorem~\ref{mainQ},
valid both for $\QQ$ and for $\QQ(t)$.

\begin{theorem}\label{th:main}
For any $\delta>0$, there is a constant $M_\delta$ with the following
property:
Let $B=(B_n)$ be an elliptic divisibility sequence arising from
a magnified $K$-rational point $P$ on an elliptic curve $E$ in minimal form.
Write $\sigma:E'\rightarrow E$ with $E'$ in minimal form, and $\sigma(P')=P$.
If $\rho(P', E')\geq\delta$, then the number
of prime power terms $B_{n}/B_1$ is bounded by $M_{\delta}$.
\end{theorem}

In light of the above-mentioned results of Hindry and Silverman,
in the case $K=\QQ(t)$, there is a uniform constant $\delta>0$
such that the condition $\rho(E', P')\geq \delta$ is always
satisfied. If Szpiro's Conjecture holds, then this is true of
$K=\QQ$ as well.

Similar methods allow a result of the following kind:

\begin{theorem}\label{th:two}
For any $\delta>0$, there is a constant $C_\delta$ with the following
property:
Let $\phi:E'\rightarrow E$ be an isogeny of minimal elliptic curves, 
then for any subgroup $\Gamma'\subseteq E'(\QQ)$ such that
$\min_{P'\in\Gamma'}\{\rho(E',P')\}\ge\delta$,
$$\#\{P\in\phi(\Gamma'):B_P\text{ is a prime power}\}\leq C_\delta^{1+\rank(\Gamma')}.$$
\end{theorem}

It follows immediately that the number of prime power terms
in the sequence $(B_n)_{n\geq 1}$
is bounded uniformly.  This observation is not
nearly as strong as
Theorem~\ref{th:main} unless one restricts attention to sequences in
which $B_1=1$, which
are in some sense rare (corresponding to integral points on elliptic curves).

\subsection{Behaviour Under Isogeny}The first lemma shows how
primes behave under isogeny, and demonstrates that the
denominators of points in the image of an isogeny admit a
canonical factorization - see~(\ref{demo}).

\begin{lemma}\label{easy-factorization} Let $E$ be a minimal elliptic
curve defined over $K$, and let
$\sigma : E'\rightarrow E$ be an isogeny of degree $m$.
Then we have
$$v\left(B_P\right)\leq v\left(B_{\sigma(P)}\right).$$
If $E'$ is also minimal, then $v(B_P)>0$ implies
$$v\left(B_{\sigma(P)}\right)\leq v\left(B_P\right)+v(m).$$
\end{lemma}
\begin{proof}
On the assumption that $E'$ is minimal at $v$, it is not hard to show
(see, for example, the exposition
in \cite{streng}) that the isogeny $\sigma$ induces a map of formal groups
$F_\sigma:\hat{E'}\rightarrow \hat{E}$
defined over $\Ocal_v$ with $F_\sigma(0)=0$ (Streng proves this for number
fields, but the proof works for any local field).
It follows immediately that if $v(x(P))<0$, as $F_\sigma(z)\in \Ocal_v[[z]]$
vanishes at $0$,
$$v(B_{\sigma(P)})=v(F_\sigma(z))\geq v(z)=v(B_P).$$
If $E$ is minimal as well, we may apply the same argument to the
dual isogeny $\hat{\sigma}:E\rightarrow E'$,
noting that the composition is $[m]$.  The argument above now tell us that
$$v(B_{\sigma(P)})\leq v(B_{mP})\leq v(B_P)+v(m).$$
\end{proof}

\begin{lemma}\label{heights}
Let $E$ and $E'$ be minimal elliptic curves defined over $K$, and
let $\sigma:E'\rightarrow E$ be an isogeny of degree $m$.
Then
$$h(E)\ll h(E')+h(m).$$
\end{lemma}

\begin{proof}
In the number field case, this is a consequence of the Normalizing
Lemma of Masser and W\"{u}stholz \cite{mw}.
Note that Masser and W\"{u}stholz use the quantity
$$w(E)=\max\{1, h(g_2), h(g_3)\}$$ to measure the `size' of
an elliptic curve, where $g_2$ and $g_3$ are the usual invariants,
but $h(E)$ as defined above satisfies
$w(E)\ll h(E)\ll w(E)$.  The argument in \cite{mw}
produces \emph{some} isomorphic copy of $E$, say $E^*/K$
such that $h(E^*)\ll h(E')+\log m$.  Since $E$ and $E^*$ are
isomorphic, and $E$ is minimal, we have $h(E)\leq h(E^*)$.

If $K$ is a function field, this property is trivial by the
boundedness of the Szpiro ratio, together with the inequality
$\deg (\Delta_{E'})\le 6\deg (\Delta_{E})$.
To see this, note first that $E$ and $E'$ have the same conductor,
by~\cite[VII Corollary 7.2]{hs}.
The assertion is then a consequence of~\cite[Theorem
0.1]{Pesenti-Szpiro}, which bounds the conductor of $E'$ in
terms of $6\deg(\Delta_{E'})$, since the degree of the conductor of $E$ is a lower bound for
$\deg (\Delta_{E})$.
Note that if we assume the $abc$ conjecture for $K$, we may do the
same thing for
number fields, but a dependence on $m$ still exists.
\end{proof}

Now we return to the main theme of this section, by showing how to
reduce the problem
to an application of Siegel's Theorem.
\begin{definition}
Given $0<\epsilon <1$, together with a constant $C>0$, say that
$P\in E(K)$ is {\it $(\epsilon, C)$-quasi-integral} if
$$h(B_P)\leq \epsilon\h(P)+C.$$
\end{definition}
Siegel's Theorem, in its strong form, is the statement that for
any $\epsilon<1$ and any constant $C$, there are only finitely
many $(\epsilon, C)$-quasi-integral points in $E(K)$. Work of
Silverman \cite{silvermansiegel} refines this claim, showing that
the number of $(\epsilon, C)$-quasi-integral points in a subgroup
$\Gamma\subseteq E(K)$ may be bounded solely in terms of
$\epsilon$, $C$, $\operatorname{rank}(\Gamma)$, and a lower bound
on $\rho(P, E)$ for non-torsion $P\in \Gamma$. Elliptic
divisibility sequences are essentially rank one subgroups of
$E(K)$, and so under Lang's Conjecture a uniform bound is obtained
on the number of $(\epsilon, C)$-quasi-integral points in an
elliptic divisibility sequence. Note that if
$$Ch(E)>\delta \h(nP),$$
then $|n|\leq\sqrt{C/(\delta \rho(P, E))}$.

In particular, applying Silverman's version of Siegel's Theorem to
a rank-one subgroup of $E(K)$ generated by some point $P$ with
$\rho(P, E)$ greater than some uniform value, some dependence of
$C$ on $E$ is acceptable (in the sense that a uniform quantitative
result is recoverable), as long as $C=O(h(E))$.

It will be demonstrated that if $P$ is the image of a $K$-rational
point under an isogeny, then $B_{nP}/B_P$ being a power of a prime
is a non-trivial quasi-integrality condition.  The function field
case indicates that one would not expect this to be true more
generally.

\begin{lemma}\label{quasiint}
Let $\sigma:E'\rightarrow E$ be an isogeny of minimal elliptic
curves over $K$, and suppose that $P=\sigma(P')$
for some $P'\in E'(K)$.  If $B_{nP}/B_P$ is a power of a single
prime, then either $nP$ is $(\epsilon_1, C_1)$-quasi-integral,
with $\epsilon_1=\frac{1}{n^2}+\frac{1}{m}$ and $C_1=O(h(E)+h(m)+h(n))$,
or $nP'$ is $(\epsilon_2, C_2)$-quasi-integral,
with $\epsilon_2=\frac{m}{n^2}$ and $C_2=O(h(E')+h(m))$.
\end{lemma}

\begin{proof}
Suppose that $B_{nP}/B_P$ is a power of a single prime.  Note
that, by Lemma~\ref{easy-factorization}, $v(B_{nP'})\leq
v(B_{nP})$ for all $v\in M_K^0$.  Suppose, for the moment, that
$v(B_{nP'})\leq v(B_P)$ for all $v\in M_K^0$.  In this case,
\begin{eqnarray*}
h(B_{nP'})&\leq& h(B_P)\\
&\leq & h(P)\\
&\leq & \h(P)+O(h(E))\\
&\leq & \frac{m}{n^2}\h(nP')+O(h(E')+h(m)).
\end{eqnarray*}

Now suppose that this is not the case, in other words,
$v(B_{nP'})>v(B_P)$ for some prime $v\in M_K^0$. Then
$v(B_{nP}/B_P)>0$, and this is the only prime for which this
happens. Furthermore, from Lemma~\ref{easy-factorization}
$$v(B_{nP}/B_P)\leq v(B_{nP})\leq v(B_{nP'})+h(m).$$
In particular,
\begin{eqnarray*}
h(B_{nP})&\leq & h(B_P)+h(B_{nP'})+h(m)+h(n)\\
&\leq & \h(B_P)+\h(B_{nP'})+O(h(E)+h(E')+h(m)+h(n))\\
&\leq & \left(\frac{1}{n^2}+\frac{1}{m}\right)\h(nP)+O(h(E)+h(m)+h(n)).
\end{eqnarray*}
\end{proof}

It has been known for a long time that sequences of the form
(\ref{eq:mersenne}) do produce new primes in a weaker sense.
Zsigmondy's Theorem \cite{zsigmondy} guarantees a new prime factor
for all terms with index~$n > 6$. The definition is made precise
now.

\begin{definition}
A nonzero term $B_n$ in a sequence $(B_n)$ of elements in $\mathcal O_K$ has a
{\it primitive divisor} $d$ if
\begin{itemize}
\item[(I)] $d$ is not a unit
\item[(II)] $d\mid B_n$
\item[(III)] $\gcd(B_m,d)$ is a unit for all $m<n$ with $B_m\neq 0$.
\end{itemize}
\end{definition}

Silverman \cite{silabc} proved the elliptic analogue of Zsigmondy's
Theorem: in any elliptic divisibility
sequence $B=(B_n)$, all the terms from some index onwards will have
a primitive divisor. This theorem has been proved in a strong
uniform manner in \cite{is} (see also \cite{emw} and \cite{pi_eds})
and and will be used crucially in
the following proof.

\begin{proof}[\sc Proof of Theorem~\ref{th:main} in the rational case.]
It will be shown that for every $\delta>0$ there is a constant
$M_\delta$ depending only on $\delta$ such that for every pair of
minimal elliptic curves $E', E/\QQ$ equipped with a non-trivial
isogeny $\sigma:E'\rightarrow E$, and every point $P=\sigma(P')$
with $\rho(P', E')\geq\delta$, there are at most $M_\delta$ values
$n$ such that $B_{nP}/B_P$ is a prime power.

The style of proof depends upon whether the isogeny is cyclic. 
Firstly, assume that $P\in kE(\QQ)$ for some integer $k\geq 2$. 
Then $P\in qE(\QQ)$ for some prime $q$, and it may be assumed that
$k$ is prime.

Suppose there are distinct non-archimedean valuations $v_1$ and $v_2$ such that
$$
v_1(B_{nkP'})>0\quad\text{and}\quad v_1(B_{kP'})=0
$$
and
$$v_2(B_{nP'})>0\quad\text{and}\quad v_2(B_{\gcd(k, n)P'})=0.$$
Then clearly $v_1(B_{nP}/B_P)>0$.  On the other hand, $v_2(B_{kP'})=0$, as
$$v_2(B_{\gcd(k, n)P'})=\min\{v_2(B_{nP'}), v_2(B_{kP'})\},$$ and so
$v_2(B_{nP}/B_P)>0$. Thus if $B_{nP}/B_P$ is a prime power, then either:
\begin{enumerate}
\item the primes diving $B_{nkP'}$ are at most those dividing $B_{kP'}$ or
\item the primes diving $B_{nP'}$ are at most those
dividing $B_{\gcd(k, n)P'}$.
\end{enumerate}
Let $Z(P', E')$ denote the set of $s$ such that $B_{sP'}$ has no
prime divisors other than those dividing $B_{tP'}$ for $t<s$. In
case (a), the term $B_{nkP'}$ in the elliptic divisibility
sequence defined by $P'$ has no primitive divisor, and so $nk\in
Z(P', E')$.   If (b) holds, and if $\gcd(k, n)<n$, then $n\in
Z(P', E')$. Thus, if $B_{nP}/B_P$ is a prime power, then either
$n\mid k$ (and recall that $k$ is assumed to be prime) or $n\in
Z(P', E')\cup\frac{1}{k}Z(P', E')$. By Theorem~7 of \cite{is}, $\#
Z(P', E')$ may be bounded solely in terms of $\rho(P', E')$ (Note
that the statement in \cite{is} is not in terms of $\rho(P', E')$,
but a simple modification of the proof shows that this is true).
Thus the number of $n$ such that $B_{nP}/B_P$ is a prime power is
bounded by some $M_{\delta,1}$ which depends only on $\delta$.

\medskip

So now suppose $P\not\in kE(\QQ)$ for any integer $k\geq 2$.  It
follows that $\sigma$ is a composition of cyclic isogenies over
$\QQ$, and so it may be assumed (without loss of generality) that
$\sigma$ itself is cyclic.  In particular, there are (by work of
Mazur \cite{mazur}) only finitely many possible values for
$m=\deg(\sigma)$.   Thus Lemma~\ref{quasiint}, replacing $h(m)$
with a sufficiently large constant, tells us that if
$B_{nP}/B_{P}$ is a prime power, then either $nP$ is $(\epsilon_1,
C_1)$-quasi-integral, with $\epsilon_1=\frac{1}{n^2}+\frac{1}{m}$
and $C_1=O(h(E)+h(n))$, or $nP'$ is $(\epsilon_2,
C_2)$-quasi-integral, with $\epsilon_2=\frac{m}{n^2}$ and
$C_2=O(h(E'))$.  If $\frac{m}{n^2}\leq\frac{3}{4}$, there is a
bound on the number of points satisfying the latter condition
depending only on $\rho(P', E')$. Clearly the number of $n$ such
that $\frac{m}{n^2}>\frac{3}{4}$ is bounded absolutely.

Now note that $\rho(P', E')\geq \delta$ and Lemma~\ref{heights} implies
that $\rho(P, E)\geq\delta_2$
for some $\delta_2>0$.  Thus, the first condition above becomes
$$h(B_{nP})\leq \left(\frac{1}{n^2}+\frac{1}{m}\right)\h(nP)+O(h(E)+h(n)).$$
Finally, for any $C$ and any $\delta_3>0$, we can
ensure that
$Ch(n)<\delta_3\h(nP)$ by taking $n$
large enough.  Thus the above becomes
$$h(B_{nP})\leq \left(\frac{1}{n^2}+\frac{1}{m}+\delta_3\right)\h(nP)+O(h(E))$$
for sufficiently large $n$, where $\delta_3$ may be taken as any positive
value. Applying the quantitative Siegel
Theorem, the proof is complete.
\end{proof}

\begin{remark}
Work of the second author \cite{pi, pi2}, in the case $K=\QQ$, also
assuming Lang's Conjecture, can be applied to
produce a uniform
bound $M$ such that for any elliptic divisibility sequence generated by a
magnified point $P$ on a minimal curve,
there are at most two values $n>M$ such that $B_{nP}$ is a prime power.
\end{remark}

In the function field case, a stronger version of
Theorem~\ref{th:main} will be proved, namely, that there is a
uniform bound $N_0$ such that $B_n/B_1$ is not a prime power for
all $n\ge N_0$. First, one lemma is required.

\begin{lemma}\label{updlemma}Let $K=\QQ(t)$ and
let $B=(B_n)$ be an elliptic divisibility
sequence arising from a non-torsion point on a minimal elliptic
curve defined over~$K$. There
is a uniform bound on the indices~$n$ for which $B_n$ fails to have a
primitive divisor.
\end{lemma}

Lemma~\ref{updlemma} will be proved in an explicit way later --
see Theorem~\ref{upd}.

\begin{proof}[\sc Proof of Theorem~\ref{th:main} in the function field
case.]
Lemma~\ref{quasiint} shows that if $B_{nP}/B_P$ is a prime power, then
$$h(B_{nP})\leq \left( \frac{1}{m}+\frac{1}{n^2}\right)\h(nP)+O(h(E))$$
or
$$h(B_{nP'})\leq\left(\frac{m}{n^2}\right)\h(nP')+O(h(E')).$$
If $n\geq \sqrt{2m}$ and $m\geq 2$, then this implies
$$h(B_{nP})\leq \frac{3}{4}\h(nP)+O(h(E))$$
or
$$h(B_{nP'})\leq \frac{1}{2}\h(nP')+O(h(E')),$$
respectively. As $\rho(E, P)$ and $\rho(E', P')$ may be bounded
below by some absolute, positive value, it follows that the number
of such $n$ must be bounded uniformly.

To show that $n$ itself must be bounded, note that if $P\in kE(\QQ(t))$ for
some integer $k\geq 2$ then the existence of a primitive divisor may be
used exactly as in the proof in the rational case -- the boundedness
of $n$ follows using Lemma~\ref{updlemma}. In the case of a cyclic
isogeny, the structure of the proof follows the rational case. The
differences are: Lang's conjecture is known to be true; a
uniform bound for the degree of the isogeny follows from the
existence of the same in the rational case since, for any isogeny over
of elliptic curves over $\QQ(t)$, some specialization will be an
isogeny of elliptic curves over $\QQ$; and Lemma~\ref{updlemma}
implies a quantitative Siegel Theorem in which the index $n$ is
bounded.
\end{proof}

That $n$ itself may be bounded in the above argument is already suggested by
Proposition~8.2 of~\cite{hindrysilverman}.  In Section~\ref{FFcase}
below we will work out some explicit bounds, in particular the
explicit bound given in Theorem~\ref{mainQt}.

\section{The Function Field Case}\label{FFcase}

\subsection{Explicit Bounds for $\QQ(t)$.}\label{explicit}

Over~$\QQ(t)$, Theorem~\ref{mainQt} raises the possibility of
finding all the irreducible terms in particular cases. The bound
of $\qtbound$ comes from applying the general version of Lang's
conjecture \cite{hindrysilverman}, assuming also the maximal
possible degree for a cyclic isogeny. As it stands, this is
useless: checking the terms coming from smaller indices involves
testing rational polynomials for irreducibility which have degrees
well in excess of one billion. In particular cases, the problem is
circumvented by obtaining a more nuanced version of Lang's
conjecture (which takes account of the reduction of small
multiples of~$P$, see Theorem~\ref{Lang-conjecture}) together with
Theorem~\ref{upc}, which highlights the dependence upon the degree
of the isogeny. This technique allows us to exhibit examples of
elliptic divisibility sequences where all of the prime (=
irreducible) terms can be computed.

\begin{example}\label{joes} Consider $$E:y^2=x^3+t^2(1-t^2)x \mbox{ with }
P=[t^2,t^2].$$
Then $B_{n}/B_1=B_n$ fails to be a prime power for all $n\ge 3$. Note that
the point~$P$ is magnified by a 2-isogeny as detailed in
\cite[Chapter 14]{cassels}.
\end{example}

On the other hand, it is worth probing the deeper difference between the cases
when $K=\QQ$ and $K=\QQ(t)$ which is probably at
work in general. It is easy to prove that the
polynomial $(t^n-1)/(t-1)$ is irreducible for every prime index~$n$.  An
analogous statement seems to be true for some elliptic divisibility sequences
over function fields. The two curves in Example~\ref{twocurves} appeared in~\cite{hindrysilverman}
as examples where the global canonical height of a point is small:

\begin{example}\label{twocurves}
The point $P=[0,0]$ on the curve
$$E:y^2+xy+(t+1)t^2y=x^3+t^2x^2$$
produces terms $B_{nP}$ which are irreducible for all primes~$n$ from
$5$ to $199$.
\smallskip
The point $P=[0,0]$ on the curve
$$E:y^2+xy+(t+1)t^3y=x^3+t^3x^2$$
produces terms $B_{nP}$ which are irreducible for all primes~$n$ from
$5$ to $199$.
\end{example}

More surprisingly, there seem to be examples where irreducibility occurs
along prime indices in a fixed residue class.
\begin{example} The point $P=[1-t,1-t]$ on the curve $$E: y^2=x^3+t(1-t)x$$ produces
irreducible terms
$B_{nP}$ whenever $n\le 79$ is a prime congruent to $3\pmod 4$.
\end{example}

That $B_{nP}$ is composite for any prime $n\equiv1\pmod 4$ is not surprising, in light of results of
Streng \cite{streng} for elliptic divisibility sequences arising from elliptic curves
(over number fields) with complex multiplication.

\subsection{Lang's conjecture}{\ }

We do {\it not} propose to make explicit the argument given in
section~\ref{both}. Instead, the technique is to work with local
heights and explicit bounds for the difference between the na\"ive
and canonical heights. When $K=\Q(t)$ the infinite place~$\infty$
corresponds to the valuation~$-\deg$. Without loss of generality,
assume that $E$ is given by a short Weierstrass form
\begin{equation}\label{minweq}
y^{2}=x^{3}+Ax+B
\end{equation}
with $A,B\in\Q [t]$,
so that $\Delta_E=-16(4A^{3}+27B^{2})$.
Then the $j$-invariant $j_E$ of
$E$ is given by
\begin{equation}\label{jdef}
j_E=\frac{-1728 (4A)^{3}}{\Delta_E}
\end{equation}
and the height $h(E)$ of the curve $E$ is
$$
h(E)=\frac{1}{12}\max (h(j_E),h(\Delta_E))=
\frac{1}{12}\max (3\deg(A),2\deg (B)) .
$$
For $v$ any place of $K$, denote by $\lambda_v(P)$ the N\'eron
local height of $P$ at $v$, and by $h_v(P)$ the na\"ive local
height
$$
h_{v} (P)=\frac{1}{2}\max \{0,-v(x(P))\}.
$$

The first lemma compares the heights of the $j$-invariant and
discriminant of $E$.
\begin{lemma}\label{inequality-deg-j-invariant-discriminant}
With $j_E$ as in (\ref{jdef}),
$$
h(j_{E})\le
\frac{3}{2}\deg(\Delta_E).
$$
\end{lemma}

\begin{proof}
Let $L$ be the smallest Galois extension of $\Q (t)$ containing
the roots $\delta_{1}$, $\delta_{2}$ and
$\delta_{3}=-\delta_{1}-\delta_{2}$ of $x^{3}+Ax+B$. Let $v$ be a
place of $L$ above the place at infinity. Since $x^{3}+Ax+B$ is
monic with polynomial coefficients,
$v_{\mathcal{P}}(\delta_{i})\ge 0$, for every finite place
$\mathcal{P}$ of $L$, so that
$v_{\mathcal{P}}(\delta_{i}-\delta_{j})\ge 0$. Summing over all
finite places of $L$, and using the fact that
$\delta_{i}-\delta_{j}$ has the same valuation at all places of
the Galois extension $L$ above the infinite place, it follows that
$v(\delta_{i}-\delta_{j})\le 0$.

Since $v$ is non-archimedean, at most one of the valuations $v
(\delta_{1}-\delta_{2})$, $v (2\delta_{1}+\delta_{2})$ and $v
(\delta_{1}+2\delta_{2})$ can be different from $\min
(v(\delta_{1}),v(\delta_{2}))$. In particular, since $\Delta_E =
((\delta_{1}-\delta_{2})(2\delta_{1}+\delta_{2})(\delta_{1}+2\delta_{2}))^{2}$,
it follows that $v(\Delta_E)\le 4\min (v(\delta_{1}),
v(\delta_{2}))$. Hence
$$
v (B^{2})=2(v(\delta_{1}) +v(\delta_{2}) + v(\delta_{1}+\delta_{2}))
\ge 6 \min (v(\delta_{1}),v(\delta_{2}))\ge \frac{3}{2}v (\Delta_E).
$$
Summing over all places $v$ of $L$ above the place at infinity gives
$$
\deg(B^{2})\le\frac{3}{2}\deg (\Delta_E).
$$
The lemma is now a consequence of the inequality
$$
\deg (A^{3}) \le
\max (\deg (B^{2}),\deg (\Delta_E))\le\frac{3}{2}\deg (\Delta_E )
$$
together with the definition $j_{E}=\frac{-1728
(4A)^{3}}{\Delta_E}$ of the $j$-invariant of $E$.
\end{proof}

The following is a version of Lang's conjecture, which was proved
in the function field case by
Hindry--Silverman~\cite{hindrysilverman}. The explicit constants
obtained here, in certain special cases, use the same basic
methods.

\begin{proposition}\label{Lang-conjecture}
There is a constant $c>0$ such that, for all non-torsion points $P\in E(K)$,
$$
c\deg (\Delta_{E}) \le \widehat{h}(P),
$$
where $c$ can be taken to be the following:
\begin{itemize}
\item in the general case, $c=10^{-9.2}$;
\item if $P$ has everywhere good reduction, $c=\frac{1}{12}$;
\item if $P$ has everywhere good reduction except at infinity,
$c=\frac{1}{16}$;
\item if $E$ does not have split multiplicative reduction (in
particular, if $E$ is isotrivial), $c=\frac{1}{1728}$;
\item if $E$ does not have split multiplicative reduction except at
infinity (in particular, if $E$ has a polynomial $j$-invariant),
$c=\frac{1}{2304}$.
\end{itemize}
\end{proposition}

\begin{proof}The first bound is the second remark on Theorem 6.1 in
\cite{hindrysilverman}. Assume now that $P$ has everywhere good reduction except maybe at
infinity. Let $v$ be a valuation different from the valuation at
infinity. The coefficients of the chosen Weierstrass equation are
polynomials. Following~\cite[III Theorem
4.5]{Lang-Diophantine-analysis}, the
N\'eron local height ${\lambda}_{v}$ at $v$ satisfies
$$
{\lambda}_{v}(P)=\frac{1}{2}\max
\{-v(x(P)),0\}+\frac{1}{12}v(\Delta_E)\ge\frac{1}{12}v(\Delta_E),
$$
while the
N\'eron local height ${\lambda}_{\infty}$ at infinity
satisfies
$$
{\lambda}_{\infty}(P)\ge
\frac{1}{2}\max \{0, \deg(X(P))\}+\frac{1}{24}\min \{0, -h(j_E)\}
\ge -\frac{1}{24}h(j_{E}),
$$
for some function $X$ on $E(K)$. Summing over all places gives
$$
\widehat{h}(P)\ge \frac{1}{12}\deg (\Delta_{E}) -\frac{1}{24}h(j_{E}).
$$
In particular, Lemma~\ref{inequality-deg-j-invariant-discriminant} implies
$$
\widehat{h}(P)\ge\left( 1-\frac{1}{4}\right)\frac{1}{12}\deg
(\Delta_{E}) =\frac{1}{16}\deg (\Delta_{E})
$$
and the third bound follows.

Now assume $P$ also has good reduction at infinity. Since the
coefficients in the Weierstrass equation are not integral at
infinity, a transformation is needed to an $\infty$-minimal
Weierstrass equation, with discriminant $\Delta_\infty$. Since the
N\'eron local height does not depend on the choice of Weierstrass
equation,
$$
{\lambda}_{{\infty}}(P)
\ge\frac{1}{12}v_\infty(\Delta_{\infty})\ge 0.
$$
Summing over all places as before,
$$
\widehat h(P)\ge
\frac{1}{12}\displaystyle\sum_{v\neq
{\infty}}v(\Delta_{E})=-\frac{1}{12}v_{\infty}(\Delta_{E})
=\frac{1}{12}\deg
(\Delta_{E})
$$
which gives the second bound.

When $E$ does not have split multiplicative reduction at $v$, the
Kodaira-N\'eron Theorem asserts that $12P$ has good reduction at
$v$. Hence the bounds for $\widehat{h}(P)$ and the functoriality
of the canonical height $\widehat{h}(12P)=144\widehat{h}(P)$ give
the final two bounds.
\end{proof}

\begin{remark} In special cases, one can improve the constants in
Theorem~\ref{Lang-conjecture} by a more careful analysis of the
reduction types at bad primes. Moreover, even when $E$ has split
multiplicative reduction, the Kodaira-N\'eron Theorem gives a good
constant also when there is a small common multiple of all the
valuations of the discriminant $\Delta_{E}$ (for example if it is
squarefree).
\end{remark}


\subsection{Na\"ive and canonical heights}{\ }

In order to estimate the difference between the na\"ive and
canonical heights of a point, pass to the local heights. The most
difficult of these is the height at infinity, which is addressed
first.

\begin{lemma}
\label{Lemma-Difference-local-height-Neron-local-height-at-infinity}
For $P\in E(K)$,
$$
-\frac56 - h(E) +\frac{1}{24}\min (0,v_{\infty}(j))\le
{\lambda}_{\infty}(P)-h_{\infty}(P) \le h(E) + \frac 56
-\frac{1}{12}\deg (\Delta_E).
$$
\end{lemma}

\begin{proof}
The coefficients $A$ and $B$ of the minimal Weierstrass
equation~\eqref{minweq} are not $v_{\infty}$-integral. A Weierstrass
equation of $E$ with
$v_{\infty}$-integral coefficients is given by
$$
Y^{2}=X^{3}+t^{-4r}AX+t^{-6r}
$$
where $r$ is the smallest integer greater than or equal to
$h(E)$. The corresponding
change of variables is given by $x=t^{2r}X$ and $y=t^{3r}Y$ and the
discriminant of this equation is
is $\Delta_{\infty}=t^{-12r}\Delta_E$.

Following \cite[III Theorem 4.5]{Lang-Diophantine-analysis}, the
N\'eron local height ${\lambda}_{\infty}$ satisfies
$$
\frac{1}{24}\min (0,v_{\infty}(j))\le
{\lambda}_{\infty}(P)-\frac{1}{2}\max (0, -
v_{\infty}(X(P))\le
\frac{1}{12}v_{\infty}(\Delta_{\infty})
$$
Since $2\lambda_\infty(P)-r\le \frac 12\max (0 ,-2r
+\deg (x(P)))\le \lambda_\infty(P)$, this can be reformulated in terms
of $x(P)$, $\Delta_E$ and $r$ as
$$
-r+\frac{1}{24}\min (0,v_{\infty}(j))\le
{\lambda}_{\infty}(P)-h_{\infty}(P) \le r
-\frac{1}{12}\deg (\Delta_E).
$$
The integer $r$ is the smallest integer greater than or equal to
$h(E)$, which is at most $h(E)+\frac56$, and the result follow.
\end{proof}

Now the difference between na\"ive and canonical heights can be
bounded.

\begin{proposition}\label{height-canonical_height-function-field-case}
For $P\in E(K)$,
$$
-\frac{3}{16}\deg(\Delta_E)-\frac 56\le\widehat{h}(P) - h(P)\le
\frac18\deg(\Delta_E)+\frac 56
$$
\end{proposition}

\begin{proof}
Following \cite[III Theorem 4.5]{Lang-Diophantine-analysis}, the
N\'eron local height ${\lambda}_{v}$ at any finite place $v$
satisfies
$$
\frac{1}{24}\min (0,v(j))\le
{\lambda}_{v}(P)-h_{v}(P)\le \frac{1}{12}v(\Delta_E).
$$
Together with
Lemma~\ref{Lemma-Difference-local-height-Neron-local-height-at-infinity},
and summing over all places, this gives
$$
-\frac 56 - h(E) -\frac{1}{24}h(j)\le
\widehat{h}(P)-h(P)\le h(E)+\frac 56.
$$
To conclude we notice that $h(j)\le 12h(E)$ and use that $h(E)\le
\frac 18\deg(\Delta_E)$ from
Lemma~\ref{inequality-deg-j-invariant-discriminant}.
\end{proof}

This subsection concludes with an estimate on the degree of the
denominator of $x(P)$, in terms of the (na\"ive) height of $P$.
This, in particular the fact that the constant $\frac 34$ is large
enough, is crucial to the explicit bounds later.

\begin{lemma}\label{compareheight} For $P\in E(K)$,
$$
\frac34h(P)-\frac3{16}\deg(\Delta_E) \le\deg(B_P).
$$
\end{lemma}

\begin{proof}
This uses a standard approach to Siegel's Theorem. Assume that
$h(P)=\frac12\deg(A_P)$ (or we are done).
Denote by $K_\infty$ the completion of $K$ at the infinite valuation,
and by $L$ the algebraic closure of $K_\infty$. Let $v$ be the
valuation on $L$ which extends the valuation $-\deg$ on $K_\infty$.
Inserting the coordinates of $P$ into the equation for $E$ gives
$$
C_P^2=(A_P-\delta_1B_P^2)(A_P-\delta_2B_P^2)(A_P-\delta_3B_P^2).
$$
This may be factored in $L$ as
$$
C_i^2=A_P-\delta_iB_P^2,
$$
because the factors of~$\Delta_E$ are squares in~$L$ (so they may
be absorbed into the~$C_i,i=1,2,3$).

As in the proof of
Lemma~\ref{inequality-deg-j-invariant-discriminant},
$v(\Delta_E)\le \frac 14v(\delta_i)$, for each $i$. If $v(A_P)\ge
v(\delta_i B_P^2)$ then we have $v (A_{P})\ge 2 v(B_{P}) +
\frac{1}{4}v(\Delta_E )$ and nothing further is needed. So
consider the case when $v(A_P)< v(\delta_i B_P^2)$ for $i=1,2,3$,
or more precisely the case when
$$
v(A_P)=2v(C_i), \quad i=1,2,3.
$$
For $i\neq j$,
$$
(\delta_j-\delta_i)B_P^2=C_i^2-C_j^2=(C_i+C_j)(C_i-C_j)
$$
therefore
$$
B_{ij+}^2=C_i + C_j \mbox{ and }B_{ij-}^2=C_i - C_j,
$$
Without loss of generality, assume $ v(B_{ij+})\le v(B_{ij-}). $

We assert that (at least) one of the following two inequalities holds:
\begin{eqnarray}
v(B_{12+})&\ge& \frac23v(B_P)+\frac12v(\delta_2-\delta_1)
\quad\mbox{ or } \notag\\
v(B_{13+})&\ge& \frac23v(B_P)+\frac12v(\delta_3-\delta_1).
\label{ineq-12+13+notboth}
\end{eqnarray}
To prove this assume that
\begin{eqnarray}
v(B_{12+})&<&\frac23v(B_P)+\frac{1}{2}v(\delta_2-\delta_1)
\quad \mbox{ and }\notag\\
v(B_{13+})&<&\frac23v(B_P)+\frac{1}{2}v(\delta_3-\delta_1).\label{both}
\end{eqnarray}
This forces
$$
v(B_{12-})>\frac13v(B_P)\quad \mbox{ and }\quad
v(B_{13-})>\frac13v(B_P).
$$
Then Siegel's relation $(C_1-C_2)+(C_2-C_3)=C_1-C_3$ forces
$$
v(B_{23-})>\frac13v(B_P).
$$
Since $\lcm(B_{12+}^{2},B_{13+}^{2})$ divides
$B_P^{2}(\delta_2-\delta_1)(\delta_3-\delta_1)$, the
inequalities~\eqref{both} imply
\begin{eqnarray*}
v(\gcd(B_{12+},B_{13+}))
&=&v(B_{12+})+v(B_{13+}) -v(\lcm(B_{12+},B_{13+})) \\
&<&\frac13v(B_P).
\end{eqnarray*}
Now Siegel's relation $(C_1+C_2)-(C_1+C_3)=C_2-C_3$ shows that
$\gcd(B_{12+},B_{13+})$ divides $B_{23-}$. In particular,
$$
\frac13v(B_P) < v(B_{23-})\le \gcd(B_{12+},B_{13+}) < \frac13v(B_P),
$$
which is absurd.

Hence one of the inequalities~\eqref{ineq-12+13+notboth} holds, say
the first. Since
$$v(B_{12-})\ge v(B_{12+}) \mbox{ and } 2C_1=B_{12+}^2+B_{12-}^2,$$
it follows that
$$
\frac12v(A_P)=v(C_1)\ge \frac43 v(B_P)+v(\delta_2-\delta_1)\ge
\frac43v(B_P)+\frac14 v(\Delta_E).
$$
To conclude, recall that $v$ extends $-\deg$.
\end{proof}


\subsection{Primitive divisors}{\ }

\begin{theorem}\label{upd} Let $P\in E(K)$ be a non-torsion point and
let $B=(B_n)$ be the elliptic divisibility
sequence arising from $P$. There
is a uniform constant~$N_0$ such that for all $n\ge N_0$, $B_n$ has a
primitive divisor.
\begin{itemize}
\item In general, $N_0=190000$.
\item If~$P$ has everywhere good reduction, $N_0=7$.
\item If~$P$ has everywhere good reduction except at infinity, $N_0=11$.
\item If~$E$ is isotrivial, $N_0=133$.
\item If~$E$ has a polynomial $j$-invariant, $N_0=151.$
\end{itemize}
\end{theorem}

\begin{proof} Suppose the term $B_n$ does not have a primitive divisor. Given
any irreducible factor $f$, there is a $d<n$ with $f|B_d$. Since
$f$ also divides $B_{\gcd(n,d)}$ we may assume $d$ is actually a
divisor of $n$. What is more, the valuation of $B_n$ and $B_d$
must be the same. It follows that if $B_n$ does not have a
primitive divisor then $B_n$ divides
$\displaystyle\prod_{p|n}B_{\frac{n}{p}}$; in particular,
$$
\deg(B_n) \le \sum_{p|n}\deg(B_{\frac{n}{p}}).
$$
Now apply Lemma~\ref{compareheight} to obtain
$$
\frac34h(nP)-\frac3{16}\deg(\Delta_E) \le
\sum_{p|n}h\left(\frac{n}{p}P\right) .
$$
Proposition~\ref{height-canonical_height-function-field-case},
with the functoriality of the canonical height, gives
\begin{eqnarray*}
&&\frac{3}{4}\left(hn^2-\frac{1}{8}\deg(\Delta_E) -\frac
56\right)-\frac{3}{16}\deg(\Delta_E) \\ &&\qquad\qquad\qquad
\le hn^2\left(\displaystyle\sum_{p|n}\frac{1}{p^2}\right)
+\frac{3}{16}\omega(n)\deg(\Delta_E) + \frac 56\omega(n),\\
\end{eqnarray*}
where $\omega (n)$ denotes the number of prime factors of $n$.
Re-arranging, dividing by $\deg(\Delta_E)\ge 1$, and applying Lang's
conjecture (Proposition~\ref{Lang-conjecture}, with $c$ equal to
the constant from
there) gives
$$
cn^2\left(\frac34-\sum_{p|n}\frac{1}{p^2}\right)\le
\frac{87+98\omega(n)}{96}.
$$
Using $\frac34-\sum_{p|n}\frac{1}{p^2}\ge 0.297$ and $\omega(n)\le
\log(n)/\log(2)$, firstly bounds~$n$ roughly. Then a more careful
analysis of small values of $n$ yields the bounds in the Theorem.
\end{proof}


\subsection{Prime powers}{\ }

Suppose now that $E'$ is another elliptic curve over $K$, defined
relative to a minimal short Weierstrass equation with discriminant
$\Delta_{E'}$, and that $\sigma:E'\to E$ is an isogeny of degree $m>1$.
The following is a direct consequence of Lemma~\ref{easy-factorization}.
\begin{lemma}\label{easy}
Suppose $P'\in E'(K)$ and $P=\sigma(P')$. Then
\begin{equation}\label{demo}
B_P=B_{P'}\widetilde B_P,
\end{equation}
with $\widetilde B_P$ coprime to $B_{P'}$.
\end{lemma}

We already know how to deal with isogenies of the form $[k]$ for
some integer~$k$, using Theorem~\ref{upd}. From now on, we work
only with cyclic isogenies. Note that the degree of such an
isogeny is bounded as in the rational case (by specialization),
$m\le\mazurbound$, by Mazur's famous result~\cite{mazur} (see also
\cite[page 265]{aec}). The explicit dependence upon the degree is
shown in what follows.

\begin{theorem}\label{upc}
Suppose $B=(B_{nP})$ is an elliptic divisibility sequence generated by
a rational point $P$ on a minimal elliptic curve. Assume~$P$ is magnified by
a cyclic isogeny of degree~$m$. There
are uniform constants~$N_0$ and $N_1$ such that for all
$$n\ge \max\{N_0, N_1\sqrt m\},$$
$B_n/B_1$ has at least
two distinct prime factors.
\begin{itemize}
\item In general, $N_0=133034, N_1=325865$.
\item If~$P'$ has everywhere good reduction, $N_0=12, N_1=5$.
\item If~$P'$ has everywhere good reduction except at infinity, $N_0=14,
N_1=6$.
\item If~$E'$ is isotrivial, $N_0=140, N_1=58$.
\item If~$E'$ has a polynomial $j$-invariant, $N_0=161, N_1=67$.
\end{itemize}
\end{theorem}

\begin{proof}[\sc Proof of Theorem \ref{mainQt}.]
The proof follows from Theorem \ref{upc} for the cyclic case and Theorem \ref{upd}
if the isogeny is~$[k]$ for some integer~$k$. The worst possible bound comes from
assuming we have a cyclic isogeny of degree~\mazurbound, together with no other special
assumptions.
\end{proof}

\begin{proof}[\sc Proof of Theorem \ref{upc}]
Write $B'=(B'_n)$ for the elliptic divisibility sequence arising from $P'$.
By Lemma~\ref{easy},
$$\frac{B_n}{B_1}.B_1=B_n'\widetilde {B_n}.
$$
If $\deg(B_n/B_1)>\deg(B'_n)>\deg(B_1)$ then
$B_n/B_1$ has two distinct prime factors.
Begin with the first inequality.
To bound $\deg(B_n/B_1)$ below use Lemma~\ref{compareheight}
and Proposition~\ref{height-canonical_height-function-field-case}. This yields
\begin{equation}\label{below}
\deg(B_n)-\deg(B_1)>h\left(\frac{3n^2}{4}-1\right)-\frac{15}{32}\deg(\Delta_E)-\frac{35}{24}.
\end{equation}
To bound $\deg(B'_n)$ above use the same tools. Note that the functoriality of
the canonical height ensures $h=mh'$ with $m\ge 2$. Then
\begin{equation}\label{above}
\deg(B'_n)<h'n^2+\frac{3}{16}\deg(\Delta_{E'})+\frac{5}{6}<\frac{h}{m}n^2+\frac{3}{96}\deg(\Delta_E)+\frac{5}{6}.
\end{equation}
The right hand side of (\ref{above}) is guaranteed to be smaller
than the right hand side of (\ref{below}) if
\begin{equation}\label{goodenough}
h\left(\frac{n^2}{m}-1\right)>\frac{1}{2}\deg(\Delta_E)+\frac{55}{24}.
\end{equation}
Applying Lang's conjecture (Proposition~\ref{Lang-conjecture},
with $c$ equal to the constant from there), $n$ is guaranteed to
be large enough if
$$
n^2 \ge 4\left(1+ \frac{67}{24c}\right).
$$
This shows that $\deg(B_n)>\deg(B_n')$ for all $n\ge N_0$, say.
Substituting in
the values for $c$ from Proposition~\ref{Lang-conjecture} gives the
stated values for $N_0$.

For the second inequality, again use Lemma~\ref{compareheight}
and Proposition~\ref{height-canonical_height-function-field-case}.
These give a lower bound for $\deg(B_n')$ of the form
$$\frac34h'n^2-\frac{9}{32}\deg(\Delta_E)-\frac{15}{24}.
$$
Similarly they give rise to an upper bound for $\deg(B_1)$ of
the form
$$\deg(B_1)<h+\frac{3}{16}\deg(\Delta_E)+\frac56.
$$
Using the relation $h=mh'$,
what we require is guaranteed if
\begin{equation}\label{alsogoodenough}
h\left(\frac{3n^2}{4m}-1\right)>\frac{15}{32}\deg(\Delta_E)+\frac{35}{24}.
\end{equation}
Using Lang's conjecture means we can guarantee $\deg(B_n')>\deg(B_1)$ if
$$n^2>m\left(1+\frac{185}{96c}\right).
$$
Inserting the various possibilities for $c$ gives $N_1$.
\end{proof}

\subsection{Explicit computations}

Given any particular example, it is possible to
compute $h, \deg(\Delta_E)$ and $m$. This data
can be fed into the bounds \eqref{goodenough} and
\eqref{alsogoodenough}.
Example \ref{joes} does not satisfy any of the
special criteria in Theorem~\ref{upc} so is best
handled this way.
The point $P=[t^2,t^2]$ on $E:y^2=x^3+t^2(1-t^2)x$ has
global canonical height equal to $\frac12$. It has bad reduction
at~$t$ and is the image of a $K$-rational point under a 2-isogeny.
This data is
inserted into
(\ref{goodenough}) and
(\ref{alsogoodenough})
to obtain a reasonable
bound for the indices $n$ beyond which $B_n$ is reducible.
Then the smaller indices can easily be checked.

On the other hand, the point
$2P=[t^4-t^2+\frac14,-t^6+\frac32t^4-\frac14t^2-\frac18]$ has good
reduction away from infinity so Theorem~\ref{upc} can be applied
directly to obtain a reasonable bound. Then the smaller indices
can be checked manually to verify that for all $n\ge 2$, the terms
$B_n/B_1=B_n$ have at least two distinct irreducible factors.

\medskip

\section{The rational case}\label{Ratcase}

Explicit results, when the base field is $\QQ$, are harder to
obtain. This is partly because Lang's conjecture is not known in
general. The tables in Section~\ref{comps} (see also Silverman
\cite[Exercise 6.10(d)]{Silverman-advance}) give examples of
rational points with very small ratio $\widehat h(P)/\log \Delta$,
forcing the constant in Lang's conjecture to be small. This small
constant plays a big role in the determination of an upper bound
on the largest index yielding a prime and it makes a reasonable
global uniform bound on the number of prime terms seem a very
distant prospect. On the other hand it is possible to use our
techniques to get reasonable bounds in particular cases, for
example, for congruent number curves:

\begin{example}\label{cncs}[{\cite{emr}}] For
$$E:y^2=x^3-N^2x$$ with $5\le N\in \mathbb N$ square-free,
assume $P$ lies in the image of a $2$-isogeny and has $x(P)<0$,
then each term $B_n/B_1$ fails to be a prime power for all $n\ge
8$. If~$P$ is integral then $B_n$ fails to be a prime power for
all $n\ge 3$. The conditions stated seem to be fulfilled for quite
a few examples with small~$N$. For example when~$N=5$ and
$P=[-4,6]$ (because $-4+5$ is a square). Similarly when $N=6$ and
$P=[-2,8]$. Frustratingly, each curve also gives rise to elliptic
divisibility sequences which are untouchable by our methods; even
to the extent that we cannot prove that only finitely many terms
are prime.
\end{example}

The heuristic argument in \cite{eds} predicts a bound for the
number of prime terms in an elliptic divisibility sequence,
provided the underlying curve is in minimal form. That argument
can be sharpened by making an assumption coming from Diophantine
analysis to predict a uniform bound; the details are given in
section~\ref{HA} below. However it cannot be adjusted to predict a
uniform bound upon the index which yields the largest prime term.
Although no such a bound is expected in general, a uniform bound
in the magnified case is not out of the question. The crucial
issue lies within the realm of Diophantine approximation and will
be discussed now.

\subsection{Heuristic argument}\label{HA}
The Prime Number Theorem suggests that the probability a large
integer $N$ is prime is roughly~$1/\log N$. Therefore, if
$(x_n)_{n\geq 1}$ is an increasing sequence of positive integers,
this suggests that (in the absence of an obvious reason to believe
otherwise), the expected number of prime terms $x_n$, with $n\leq
X$, is approximately
\begin{equation}\label{heur}
\sum_{n\leq X}\frac{1}{\log x_n}.\end{equation}
In particular, one should suspect that $x_n$ is prime infinitely
often if and only if the sum in (\ref{heur})
diverges as $X\rightarrow \infty$.  If $x_n=(a^n-b^n)/(a-b)$,
then $$\log x_n\sim nh(a/b)\text{ as }n\rightarrow\infty,$$
and so the sum diverges like a constant multiple of $\log X$.
All the available evidence supports the belief that a sequence
of this form with $a$ and $b$ not both $k^{\rm th}$ powers of integers,
for some $k>1$,
will produce prime terms at that rate.
If $x_n=B_n/B_1$, on the other hand,
$$\log x_n\sim n^2\h(P)\text{ as }n\rightarrow\infty,$$
where $\h(P)$
is the global canonical height of $P$.
As a consequence, the sum in (\ref{heur}) converges,
and it seems likely that $B_n/B_1$ is prime only finitely often.

A sharpened version of this heuristic argument can be given. Using
David's Theorem \cite{david} from elliptic transcendence theory,
it was argued in \cite{pe} that a lower bound for $\log B_n$ of
the following kind holds:
\begin{equation}\label{david}
\log B_n > hn^2 - C\log n(\log \log n)^4.
\end{equation}
The dependence of the constant $C$ is interesting. Suppose $E$
is minimal and $C=O(h(E))$. Then Lang's conjecture implies that
the sum in (\ref{heur}) is uniformly bounded above: in other
words, for elliptic divisibility sequences coming from curves
in minimal form, the sequence with $n^{\rm th}$ term $B_n/B_1$ should
be a prime a uniformly bounded number of times. Note that this
argument does not suggest a uniform bound upon the largest index
$n$ which yields a prime term $B_n/B_1$. Although David's
Theorem is very strong, currently the best form of the
constant $C$ has a polynomial but non-linear
dependence upon $h(E)$.
If a version of (\ref{david}) could be proved with a linear dependence
on $h(E)$ then, together with Lang's conjecture, we obtain a heuristic justification for the
uniform primality conjecture in general.

Failing this, a form of (\ref{david}) with a smaller main term but
a linear dependence in the error term would be adequate. This
technique has been used a number of times. This was precisely the
issue in the case when $K=\QQ(t)$. In this case, a lower bound for
$\deg(B_P)$ in terms of the height was given in
Lemma~\ref{compareheight}. This is weaker than the equivalent form
of (\ref{david}) in that the leading term is only $\frac34$ of
what it might be. On the other hand, the error term is linear in
$h(E)$ (or $\deg (\Delta_E)$). A similar device was exploited
in~\cite{emw} to obtain a uniform primitive divisor theorem.
In~\cite{emw}, the main term in (\ref{david}) was replaced by a
term only $\frac14$ of what it might be. Any attempt to prove
Conjecture~\ref{Qupc} is likely to encounter this phenomenon.

\subsection{Explicit computations}\label{comps}
Here are presented prime terms in sequences generated by the
rational points with smallest known height, in 18 cases drawn from
Elkies' tables. The computations were performed using
Pari-GP~\cite{parigp} (for the first 6 cases) and
MAGMA~\cite{magma} (for the other 12). In the table,
\begin{itemize}
\item $E$ is an elliptic curve given by a minimal equation as a
vector $[a_1,\dots ,a_6]$ in Tate's notation; \item $P$ denotes a
non-torsion point in $E(\Q)$;
\item $(B_{n})_{n\in\N}$ denotes the elliptic divisibility sequence associated to $P$;
\item $h_0$ denotes 1000 times the global canonical height of $P$
\item $c_0$ denotes 10000 times the ratio ${\widehat h(P)}/{h(E)}$
\item $N_{2}$ denotes the maximal index for which the primality of $B_{n}$ has been tested;
\item $N_{0}$ denotes the number of indices $n\le N_{2}$ such that $B_{n}=1$;
\item $N_{1}$ denotes the number of indices $n\le N_{2}$ such that $B_{n}$ is a prime number;
\item $N_{3}$ is the greatest index $n\le N_{2}$ such that $B_{n}$ is a prime number.
\end{itemize}
The largest prime obtained comes from the first pair: the point
$$
P=[7107, 594946]
$$
on the curve
$$
E:y^2+xy+y=x^3+x^2-125615x+ 61201397
$$
yields a prime term $B_{3719}$ with $26774$ decimal digits.

\newpage
\begin{center}
\begin{tabular}{|l|l|l|l|l|l|l|l|}
\hline &&&&&&&\\[-10pt]
\textrm{Minimal model} & \textrm{Point $P$} & $h_0$ & $c_0$ & $N_{0}$ & $N_{1}$& $N_{2}$ & $N_{3}$\\[3pt]
\hline
\hline
$[1,1,1,-125615,61201397]$&$[7107, 594946]$&4.45&1.06&15&32&4500&3719\\
\hline
$[1,0,0,-141875,18393057]$&[-386, -3767]&4.51&1.18&15&32&4900&1811\\
\hline
$[1,-1,1,-3057,133281]$&[591, -14596]&4.86&1.65&13&29&4700&541\\
\hline
$[1,1,1,-2990,71147]$&[27, -119]&4.98&1.84&14&32&4500&829\\
\hline
$[0,0,0,-412,3316]$&[-18, -70]&5.63&2.90&11&28&4300&317\\
\hline
$[1,0,0,-4923717,4228856001]$&[1656, -25671]&5.71&1.24&14&29&4300&419\\
\hline
$[1,0,0,-13465,839225]$ & [80, 485] & $5.77$ & 
$17.2$ & $15$ & $25$ & $3000$ & $571$\\
\hline
$[1,0,0,-21736,875072]$ & $[-154, -682]$ & $5.92$ & 
$17.1$ & $12$ & $25$ & $3000$ & $953$\\
\hline
$[1,-1,1-1517,26709]$ & $[167, -2184]$ & $6.03$ & 
$25.7$ & $13$ & $22$ & $3000$ & $1283$\\
\hline
$[1,0,0,-8755,350177]$ & $[14, 473]$ & $6.12$ & 
$18.9$ & $15$ & $32$ & $3000$ & $401$\\
\hline
$[1,-1,1,-180,1047]$ & $[-1, 35]$ & $6.42$ & 
$37.1$ & $12$ & $27$ & $3000$ & $383$\\
\hline
$[1,0,0,-59852395,185731807025]$ & $[12680, 1204265]$ & $6.56$ & 
$12.0$ & $13$ & $21$ & $3000$ & $359$\\
\hline
$[1,0,0,-10280,409152]$ & $[304, -5192]$ & $6.62$ & 
$20.2$ & $13$ & $32$ & $3000$ & $311$\\
\hline
$[0,1,1,-310,3364]$ & $[-19, 52]$ & $6.70$ & 
$36.7$ & $12$ & $20$ & $3000$ & $103$\\
\hline
$[1,0,0,-42145813,105399339617]$ & $[31442, 5449079]$ & $6.78$ & 
$14.8$ & $14$ & $21$ & $3000$ & $349$\\
\hline
$[1,0,0,-25757,320049]$ & $[-116, -1265]$ & $6.82$ & 
$19.4$ & $14$ & $22$ & $3000$ & $83$\\
\hline
$[1,0,0,-350636,80632464]$ & $[352, 748]$ & $6.91$ & 
$16.6$ & $13$ & $22$ & $3000$ & $137$\\
\hline
$[1,0,0,-23611588,39078347792]$ & $[-3718, -272866]$ & $7.41$ & 
$16.6$ & $13$ & $23$ & $3000$ & $109$\\
\hline
\end{tabular}
\end{center}


\end{document}